\documentclass[a4paper]{amsart}

\usepackage[T1]{fontenc}
\usepackage{amsmath}
\usepackage{amsthm}
\usepackage[style=alphabetic-verb, backend=biber, maxalphanames=4, maxbibnames=99]{biblatex}
\usepackage{pdfpages}
\usepackage{hyperref}
\usepackage{todonotes}
\usepackage{amsfonts}
\usepackage{amssymb}
\usepackage{float}
\usepackage{verbatim}

\usepackage[nameinlink,capitalize]{cleveref}

\usepackage{graphicx}
\usepackage{stackengine}
\stackMath
\newcommand{\precdoteq}{%
  \mathrel{\stackunder[0.1ex]{\precdot}{\rule{1.2ex}{0.1ex}}}%
}

\graphicspath{ {./fig} }
\makeatletter
\theoremstyle{plain}

\newtheorem{thm}{\protect\theoremname}[section]
\theoremstyle{plain}
\newtheorem{lem}[thm]{\protect\lemmaname}
\newtheorem{prop}[thm]{Proposition}
\newtheorem{cor}[thm]{Corollary}

\newtheorem{fct}[thm]{Fact}
\theoremstyle{definition}
\newtheorem{defn}[thm]{\protect\definitionname}
\newtheorem{exs}[thm]{Examples}
\newtheorem{ex}[thm]{Example}
\newtheorem{question}[thm]{\protect\questionname}

\newtheorem{que}[thm]{Question}
\newtheorem{rem}[thm]{Remark}
\newtheorem{nota}[thm]{Notation}

\crefname{prop}{Proposition}{Propositions}
\Crefname{prop}{Proposition}{Propositions}
\crefname{cor}{Corollary}{Corollaries}
\Crefname{cor}{Corollary}{Corollaries}
\crefname{fct}{Fact}{Facts}
\Crefname{fct}{Fact}{Facts}
\crefname{thm}{Theorem}{Theorems}
\Crefname{thm}{Theorem}{Theorems}

\providecommand{\set}[2]{\{#1 : #2\}}

\usepackage{tikz-cd}
\makeatother
\providecommand{\definitionname}{Definition}

\providecommand{\lemmaname}{Lemma}
\providecommand{\questionname}{Question}
\providecommand{\theoremname}{Theorem}
\providecommand{\theoremname}{Remark}

\DeclareMathOperator{\VC}{VC}
\newcommand{\dom}
{\mathord{\operatorname{dom}}}
\DeclareMathOperator{\Min}{MIN}
\newcommand{\predom}{\mathord{\operatorname{predom}}}
\DeclareMathOperator{\fin}{fin}
\DeclareMathOperator{\VCcof}{VCcof}
\DeclareMathOperator{\cof}{cof}
\DeclareMathOperator{\otp}{otp}

\DeclareMathOperator{\im}{Im}
\DeclareMathOperator{\cl}{cl}
\DeclareMathOperator{\Inf}{Inf}
\newcommand{\Ord}{\operatorname{\mathbf{Ord}}}
\newcommand{\Ff}{\mathcal{F}}
\usepackage{graphicx}

\newcommand{\precdot}{\prec\mathrel{\mkern-5mu}\mathrel{\cdot}}

\makeatletter
\newcommand{\preceqdot}{\mathrel{\mathpalette\pr@ceqd@t\relax}}
\newcommand{\pr@ceqd@t}[2]{%
  \begingroup
  \sbox\z@{$#1\prec$}\sbox\tw@{$#1\preceq$}%
  \dimen@=\dimexpr\ht\tw@-\ht\z@\relax
  {\preceq}%
  \mkern-5mu
  \raisebox{\dimen@}{$\m@th#1\cdot$}%
  \endgroup
}

\makeatother

\author{Omer Ben-Neria, Itay Kaplan, George Peterzil}
\title{Cofinal families of finite VC-dimension}

\thanks{\\
The first author would like to thank the Israel Science Foundation for partial support of this research (Grant 1302/23).\\
The second author would like to thank the Israel Science Foundation
for partial support of this research (Grant no. 804/22). }

\address{Omer Ben Neria \\
The Hebrew University of Jerusalem\\
Einstein Institute of Mathematics \\
Edmond J. Safra Campus, Givat Ram\\
Jerusalem 91904, Israel}
\email{omer.bn@math.huji.ac.il}

\urladdr{https://math.huji.ac.il/~omerbn/}

\address{Itay Kaplan \\
The Hebrew University of Jerusalem\\
Einstein Institute of Mathematics \\
Edmond J. Safra Campus, Givat Ram\\
Jerusalem 91904, Israel}
\email{kaplan@math.huji.ac.il}
\urladdr{https://itaykaplan.huji.ac.il/}
\address{George Peterzil \\
The Hebrew University of Jerusalem\\
Einstein Institute of Mathematics \\
Edmond J. Safra Campus, Givat Ram\\
Jerusalem 91904, Israel}
\email{george.peterzil@mail.huji.ac.il}
\urladdr{https://sites.google.com/mail.huji.ac.il/gpeterzil
}


\begin{document}
\maketitle
\begin{abstract}
    Given infinite cardinals $\theta\leq \kappa$, we ask for the minimal VC-dimension of a cofinal family $\Ff\subseteq[\kappa]^{<\theta}$. We show that for $\theta=\omega$ and $\kappa=\aleph_n$ it is consistent with ZFC that there exists such a family of VC-dimension $n+1$, which is known to be the lower bound. For $\theta>\omega$ we answer this question completely, demonstrating a strong dichotomy between the case of singular and regular $\theta$. We furthermore answer some relative and generalized versions of the above question for singular $\theta$, and answer a related question which appears in \cite{BBNKS}.
\end{abstract}
\section{Introduction}

Suppose that $X$ is an infinite set of size $\kappa$ and that $\Ff$ is a family of subsets, each of cardinality at most $\theta$ for $\theta\leq \kappa$. Consider two possible properties: 
\begin{enumerate}
    \item \label{enu:cofinal} The family $\Ff$ is cofinal: every set of cardinality $<\theta$ is contained in a member of $\Ff$.
    \item \label{enu:small VC} The family $\Ff$ has finite 
    VC-dimension, meaning that there is some $n<\omega$ such that given any subset $s$ of $X$ of size $n$, there is some $t\subseteq s$ such that $t$ is not of the form $s \cap A$ for some $A\in \Ff$, see \cref{VCdim}. 
\end{enumerate}

Observe that (\ref{enu:cofinal}) pulls $\Ff$ towards having many sets while (\ref{enu:small VC}) pulls $\Ff$ towards having few sets. This precise tension is the center of this paper. Generalizing (\ref{enu:cofinal}) we may introduce a third cardinal $\lambda\leq \theta$ and ask that $\Ff$ is $\lambda$-cofinal, meaning that every set of size $<\lambda$ is contained in a member of $\Ff$. This leads us to:



\begin{que}
\label{main}
    Given cardinals $\lambda\leq\theta\leq\kappa$, what is the minimal VC-dimension of a $\lambda$-cofinal family $\Ff\subseteq[\kappa]^{<\theta}$ (where $[\kappa]^{<\theta}$ is the set of all subsets of $\kappa$ of cardinality $<\theta$)?
\end{que}

Consider the family of proper initial segments of $\omega$. Every finite subset of the natural numbers
is contained in such an initial segment, and it can be seen to have VC-dimension 1. Thus, we have that for $\lambda=\kappa=\theta=\omega$, the answer to \cref{main} is 1. 

In \cite{BBNKS}, Bays, Simon and the first two authors, motivated by questions in model theory, considered a version of \cref{main}. Generalizing the example of initial segments of $\omega$, they proved:

\begin{fct}{\cite[Theorem 3.8]{BBNKS}}
\label{omega_1}
    There is an $\omega$-cofinal family $\Ff\subseteq[\omega_1]^{<\omega}$ with $\VC(\Ff)=2$.
\end{fct}

The following result in their paper demonstrates that the VC-dimension for such families cannot be smaller, and more generally gives a lower bound for a possible answer to \cref{main}.

\begin{fct}{\cite[Proposition 3.3 and Remark 3.7]{BBNKS}}
\label{lower}
    For $n<\omega$, a cardinal $\kappa$ and a family $\Ff\subseteq[\kappa^{+n}]^{<\kappa}$ which is $(n+2)$-cofinal we have $\VC(\Ff)\geq n+1$. 
\end{fct}

\begin{rem}
    \cref{lower} does not appear exactly in this form in \cite{BBNKS}: there, the family is required to cover every finite subset of $\kappa^{+n}$, not just those of size $n+1$. However, the proof there gives this stronger formulation. 
\end{rem}

In this paper, we attempt to continue this project. In particular, we ask when the bound of \cref{lower} is optimal. The following is the first result proved.

\begin{thm}[\cref{thm:main}(\ref{enu:regular},\ref{enu:singular})]
\label{uncountable}
    Fix infinite cardinals $\lambda\leq \theta\leq \kappa$.
    \begin{enumerate}
        \item  If $\theta$ is regular and uncountable, then there is a $\lambda$-cofinal $\Ff\subseteq[\kappa]^{<\theta}$ with $\VC(\Ff)\leq n$ if and only if $\kappa<\theta^{+n}$.
        \item If $\theta$ is singular and $\cof(\theta)<\lambda$, then every $\lambda$-cofinal $\Ff\subseteq[\kappa]^{<\theta}$ has $\VC(\Ff)=\infty$.
    \end{enumerate}
\end{thm}

In light of \cref{omega_1,lower,uncountable}, the following two questions are therefore natural:

\begin{que}${}$
\begin{enumerate}
    \item What can we say for $\theta=\omega$ and $\kappa=\omega_n$ for $n>1$?
    \item What can we say for singular $\theta$, $\kappa=\theta^{+n}$ and $\lambda\leq \cof(\theta)$?
\end{enumerate}    
\end{que}

Regarding those two questions, we get the following consistency results.

\begin{thm}[\cref{thm:main}(\ref{enu:omega},\ref{enu:measurable})]
\label{consistency}${}$
    \begin{enumerate}
        \item For every $n<\omega$, it is consistent with ZFC that there is an $\omega$-cofinal $\Ff\subseteq[\omega_n]^{<\omega}$ with $\VC(\Ff)=n+1$.
        \item Relative to the existence of a measurable cardinal, it is consistent that   there is a singular cardinal $\theta$ of countable cofinality such that for every $n<\omega$ there is some $\omega$-cofinal $\Ff\subseteq[\theta^{+n}]^{<\theta}$ such that $\VC(\Ff)=n+1$.
    \end{enumerate}
\end{thm}

The notion of VC-dimension is important in machine learning, where it corresponds to \emph{PAC} learnability (see \cite{MLShaiShai}), and also has a model theoretic counterpart, namely the \emph{NIP/IP} dichotomy, which we will not go into here, but see \cite{Simon} for more. An interesting application is the existence of a plethora of examples of 2-sorted structures with IP, whose language each has a single binary NIP relation. 

In their paper, the following was also proved:
\begin{fct}{\cite[Theorem 4.1]{BBNKS}}
\label{IP}
    Suppose $\theta$ is an infinite cardinal, $X$ a set with $|X|\geq \theta^+$ and $\Ff\subseteq[X]^{<\theta}$ is $\omega$-cofinal, then the structure $\mathcal{M}=(X,\Ff,\in)$ has $IP$, i.e., there is some definable family in $\mathcal{M}$ with IP (see \cref{VCdim}). 
\end{fct}

Since the structure $(X,\Ff,\in)$ is interpretable in every structure in which $\Ff$ is definable, \cref{uncountable,consistency} together with \cref{IP} give us a family of examples (consistently) of NIP families of sets which are not definable in any NIP structure, which would hopefully give us more understanding of local NIP phenomena, and perhaps serve as a source for counterexamples to conjectures regarding such formulas.

This work is part of the third author's master's thesis. The complementary work towards this thesis was written in \cite{Fin} together with Johanna Steinmeyer, in it a finitary analogue of the above question is considered. 

The third author would like to thank the first two authors for guiding him towards the completion of this thesis, and to Uri Abraham for refereeing the thesis with great care, leading to improvements in its presentation, and therefore in the presentation of this paper.

\subsection*{A bit on the proofs.}

The main construction throughout this paper is that of an $n$-ordering system, see \cref{def:n-ordering system}. The idea is to construct systems of compatible well-orders on a cardinal and its initial segments (and each of their initial segments, etc.) and considering sets which are closed relative to these systems, in a certain precise sense (see \cref{def:closed sets}). This idea originated in \cite[Theorem 3.8]{BBNKS} where this was done for $n=1$. Having this general definition forces us to prove many properties by induction on $n$. Accordingly, we made an effort to find the correct definitions so that the proofs are as simple as possible.



\section{Preliminaries and ordering systems}

\subsection{VC dimension and cofinal families of sets}
\begin{defn}
\label{cofinal}
    Given a set $X$ and a cardinal $\theta$, write $[X]^{<\theta}$ for the collection of subsets of $X$ of size less than $\theta$.
    Say that $\Ff\subseteq[X]^{<\theta}$ is \emph{$\lambda$-cofinal} if every $A\in[X]^{<\lambda}$ is contained in a member of $\Ff$.
\end{defn}

\begin{defn}
\label{VCdim}
    Fix a family $\Ff\subseteq\mathcal{P}(X)$ and a cardinal $\kappa\leq |X|$.
    \begin{enumerate}
        \item Say that $A\subseteq X$ is \emph{shattered by $\Ff$} if for every $A_0\subseteq A$ there is some $S_0\in\Ff$ such that $A_0=A\cap S_0$.
        \item The \emph{VC-dimension of $\Ff$} is the maximal size of a finite subset of $X$ which is shattered by $\Ff$, or $\infty$ if this is unbounded. Say that $\Ff$ is a \emph{VC-class} if $\VC(\Ff)<\infty$.  
    \end{enumerate} 
\end{defn}

\begin{defn}\label{def:VC cof}
    Given cardinals $\lambda\leq\theta\leq\kappa$, let $\VCcof(\lambda,\theta,\kappa)$ be the least number $n\in\omega\cup\{\infty\}$ such that there is a  $\lambda$-cofinal family $\Ff\subseteq[\kappa]^{<\theta}$ of VC-dimension $n$.
\end{defn}

We start by investigating basic properties of this function.

\begin{prop}
\label{monotonicity} 
    Given $\Ff\subseteq \mathcal{P}(X)$ and $X_0\subseteq X$ we have that any subfamily $\Ff_0$ of the family $\Ff\restriction X_0=\{X_0\cap S: S\in\Ff\}$ satisfies $\VC(\Ff_0)\leq \VC(\Ff)$. In particular, for every choice of cardinals $\lambda\leq\theta\leq\kappa$ and $\lambda_0\leq\theta_0\leq\kappa_0$ with $\kappa_0\leq\kappa$, $\lambda_0\leq\lambda$ and $\theta\leq\theta_0$, we have that $\VCcof(\lambda_0,\theta_0,\kappa_0)\leq \VCcof(\lambda,\theta,\kappa)$.
\end{prop}

\begin{proof}
    The first part is trivial. For the second part, let $\Ff\subseteq[\theta]^{<\kappa}$ be $\lambda$-cofinal, and apply the first part with $X=\theta$ and $X_0=\theta_0$. It is clearly still $\lambda$-cofinal, and consists of sets of size less than $\kappa \leq \kappa_0$, so in particular $\lambda_0$-cofinal, giving us the desired inequality.
\end{proof}

Restating \cref{lower} in these terms, we have:

\begin{fct}{\cite[Proposition 3.3 and Remark 3.7]{BBNKS}}
\label{lower1}
    For every infinite cardinal $\kappa$ and natural number $n$, we have $\VCcof(n+2,\theta,\theta^{+n})\geq n+1$.
\end{fct}


From \cref{lower1} we get the following interesting corollary, which is not used elsewhere in this work. It also follows more directly from \cite[Theorem 2.2]{Fin}.

\begin{cor}
    Suppose $\Ff\subseteq[\omega]^{<\omega}$ has the following ``uniform cofinality'' property: there is a function $f:\omega\rightarrow \omega$ such that given $A\in[\omega]^n$ there is $B\in[\omega]^{<f(n)}$ with $A\subseteq B$ and $B\in\Ff$. Then $\VC(\Ff)=\infty$.
\end{cor} 

\begin{proof}
    Fix such a family $\Ff$ with a corresponding function $f:\omega\rightarrow \omega$. consider the two-sorted structure $\mathcal{M}=(\omega,\Ff,\in)$, where $\mathord{\in} \subseteq \omega\times\Ff$ is the incidence relation. Then for every $n<\omega$, we have that $\mathcal{M}\models \psi_n$, where:
\begin{gather*}
   \psi_n:= \ \forall x_0,\ldots,x_{n-1}\in \omega\ \exists y_0,\ldots,y_{f(n)-1}\in\omega\ \exists S\in\Ff \\
   \big\{x_0,\ldots,x_{n-1}\big\}\subseteq S\land \big\{y_0,\ldots,y_{f(n)-1}\big\}=S
\end{gather*}

    Assume towards a contradiction that $\VC(\Ff)=m<\omega$. Let $\mathcal{N}=(N,\hat{\Ff})$ be an elementary extension of cardinality $\aleph_m$. Without loss of generality we have $N=\aleph_m$. Now, write $\hat{\Ff}_{\fin}=\hat{\Ff}\cap[\aleph_m]^{<\omega}$. Since having VC-dimension $m$ is a first-order property, we have $\VC(\mathcal{\hat{F}})=m$, hence $\VC(\mathcal{\hat{F}}_{\fin})\leq m$. Because $\mathcal{N}\models \psi_n$ for every $n<\omega$, we have that $\hat{\Ff}_{\fin}$ is cofinal as well, a contradiction to the bound of \cref{lower1}.
\end{proof}

The following example is a strong hint towards the ideas to follow.

\begin{exs}
\label{order}
    \begin{enumerate}
        \item We claim that $\VCcof(\aleph_0,\aleph_0,\aleph_0)=1$. First, the family $\Ff=\big\{\{0,\ldots,n-1\}: n<\omega\big\}$ witnesses $\VCcof(\aleph_0,\aleph_0,\aleph_0)\leq 1$. Indeed, given any finite nonempty $A\subseteq \omega$ we have $A\subseteq\sup A+1$, and given any two distinct $n_1<n_2<\omega$, the set $\{n_1,n_2\}$ is not shattered by $\Ff$, as any $I\in\Ff$ with $n_2\in I$ has $n_1\in I$ as well, so $\VC(\Ff)<2$. Since any family with at least two distinct sets has VC-dimension at least 1, we have $\VC(\Ff)=1$. \cref{monotonicity,lower1} imply that $1\leq \VCcof(\aleph_0,\aleph_0,\aleph_0)$ as well.
        \item In fact, given any infinite cardinal $\kappa$, this generalizes to demonstrate that $\VCcof(\cof(\kappa),\kappa,\kappa)=1$, with the upper bound witnessed by the family of proper initial segments of $\kappa$. 
        \item Note that for every natural number $n$ and infinite cardinals $\theta\leq \kappa$, the family $[\kappa]^{n+1}$ witnesses $\VCcof(n+2,\theta,\kappa)\leq n+1$. In fact, applying \cref{lower1} and monotonicity gives us an equality whenever $\kappa\geq\theta^{+n}$. 
    \end{enumerate}
\end{exs}

Our main method for constructing cofinal families will use the following, which was implicitly constructed in \cite{BDHSY} and explicitly for the case $n=2$ in \cite{BBNKS}

\begin{defn}\label{def:n-ordering system}
    By induction on $0 < n < \omega$ we define when a function $\mathord{\precdoteq}:[X]^{<n}\rightarrow\mathcal{P}(X^2)$ is an \emph{$n$-ordering system} on a set $X$. In the following and throughout, we write $\precdoteq_s$ for $\mathord{\precdoteq}(s)$ and $\precdot$ for the irreflexive variant, i.e., $\precdot_s := \precdoteq_s \setminus \set{(x,x)}{x\in X}$.
    
    \begin{itemize}
        \item For $n=1$, $\precdoteq$ is a 1-ordering system if $\mathord{\precdoteq}_{\emptyset}$ is a (non-strict) well-order on $X$. 
        \item For $n>1$ and any $X$, $\precdoteq$ is an $n$-ordering system if it satisfies the following conditions:
        \begin{itemize}
            \item $\precdoteq_\emptyset$ is a (non-strict) well-order on $X$.
            
            \item  For any $x_0 \in X$, the function $\mathord{\precdoteq}^{x_0}:[X_{\precdot_\emptyset x_0}]^{<n-1}\rightarrow\mathcal{P}(X^2)$ given by $\mathord{\precdoteq}^{x_0}_t = \mathord{\precdoteq_{t\cup \{x_0\}}}$ 
            is an $(n-1)$-ordering system on the $\precdoteq_\emptyset$-initial segment $X_{\precdot_\emptyset x_0} := \set{x\in X}{x\precdot_\emptyset x_0}$.
        \end{itemize}
    \end{itemize}
    We also define the domain of $\precdot_s$ (or $\precdoteq_s)$, denoted by $\dom(\precdot_s)=\dom(\precdoteq_s)$, as usual: $\set{x\in X}{(x,x) \in \mathord{\precdoteq_s}}$.
    \end{defn}

As an illustration of the idea, consider the following instance of the above definition.

    \begin{ex}
        A 3-ordering system on a set $X$ consists of the following:
        \begin{itemize}
            \item A well-order $\precdot_\emptyset$ on $X$.
            \item For every $x\in X$, a well-order $\precdot_{\{x\}}$ on the set $X_{\precdot_\emptyset x}$ defined above.
            \item For every $x,y\in X$ such that $y\precdot_\emptyset x$, a well-order $\precdot_{\{x,y\}}$ on the set $\{z\in X: z\precdot_\emptyset x\land z\precdot_{\{x\}}y\}$.
        \end{itemize}
    \end{ex}

    \begin{rem}\label{rem:restriction to smaller n}
        By induction on $n$, if $\precdot$ is an $n$-ordering system on a set $X$ and $k<n$, then $\mathord{\precdot}\restriction [X]^{<k}$ is a $k$-ordering system.
    \end{rem}

    \begin{defn} \label{def:bound on order type}
        Given an $n$-ordering system $\precdot$ on a set $X$ and a finite sequence of ordinals $\bar{v}=(v_0,\ldots,v_{n-1})\in \Ord^n$, say that $\bar{v}$ \emph{bounds the order type of $\precdot$} if for every $s\in[X]^{<n}$ we have $\otp(\precdot_s)\leq v_{|s|}$.
    \end{defn}

Our goal is to use $n$-ordering systems for various $n < \omega$ to produce cofinal families of (relatively) small VC dimension. 

\begin{rem}\label{rem:domain of <<s union a}
     Suppose $\precdot$ is an $n$-ordering system on a set $X$ for $1<n$. Then, for every set $s \in [X]^{<n-1}$ and every $a\in X$ such that $a$ is in the domain of $\precdot_s$, the domain of $\precdot_{s\cup \{a\}}$ is $X_{\precdot_s a} = \set{x \in X}{x \precdot_s a}$. 
          The proof is by induction on $1<n$. If $s$ is empty this follows directly from \cref{def:n-ordering system}. This is the only case when $n=2$ so this covers the base of the induction. Otherwise, let $x_0 = \max_{\precdot_\emptyset} s$ and let $t=s \setminus \{x_0\}$. Then the domain of $\precdot_{s \cup \{a\}}$ is the same as the domain of $\precdot^{x_0}_{\{a\} \cup t}$. Since $\precdot^{x_0}$ is an $(n-1)$-ordering system and $1<n-1$ (since $s$ is nonempty), by induction said domain is $\set{x \in X_{\precdot_{\emptyset} x_0}}{x \precdot^{x_0}_t a}$. But clearly, if $x \precdot^{x_0}_t a$ then $x \in X_{\precdot_{\emptyset} x_0}$ (since $\precdot^{x_0}$ is an $(n-1)$-ordering system on $X_{\precdot_{\emptyset} x_0}$), so said domain is $\set{x \in X}{x \precdot^{x_0}_t a}$ which is by definition $\set{x \in X}{x \precdot_s a}$ as required.

\end{rem}

\begin{defn} \label{def:closed sets}
    Suppose $\precdot$ is an $n$-ordering system on a set $X$. A set $S \subseteq X$ is \emph{$\precdot$-closed} if for every $s\subseteq S$ of size $|s|=n-1$ and every $a,b\in \dom(\mathord{\precdot_s})$, if $b\in S$ and $a \precdot_s b$ then $a\in S$.
        
\end{defn}
\begin{ex}\label{exa:n=1}
    If $\precdot$ is a 1-ordering system, then a set is $\precdot$-closed if and only if it is an initial segment of $\precdot_\emptyset$.
\end{ex}

\begin{defn} \label{def:theta-closures}
    Let $\precdot$ be an $n$-ordering system on a set $X$.
    \begin{itemize}
        \item Given $\lambda\leq\theta$, say that $\precdot$ \emph{has $(\lambda,\theta)$-closures} if every $A\in[X]^{<\lambda}$ is contained in an $\precdot$-closed set $B\in[X]^{<\theta}$.
        \item Say that $\precdot$ \emph{has $\theta$-closures} whenever it has $(\theta,\theta)$-closures. Say that it \emph{has finite closures} if it has $\omega$-closures.
    \end{itemize}
\end{defn}

\begin{lem}
\label{Lemma: inductive maximum}
    Suppose $\precdot$ is an $n$-ordering system on a set $X$ and $s\in[X]^{<\omega}$. Then for every $k\leq \min\{n-1,|s|\}$ there is a unique $s_{k}\in[s]^k$ such that $s\setminus s_{k}\subseteq \dom(\precdot_{s_{k}})$.
\end{lem}

\begin{proof}
    We prove this lemma by induction, where the case $n=1$ is trivial, as is the case of $s=\emptyset$. For $n>1$ and $s\neq\emptyset$ we write $x_0=\max_{\precdot_\emptyset}s$. By the induction hypothesis there exists $t_{k-1}\in[(s\setminus\{x_0\})]^{k-1}$ such that $s\setminus(\{x_0\}\cup t_{k-1})\subseteq \dom(\precdot^{x_0}_{t_{k-1}})=\dom(\precdot_{t_{k-1}\cup\{x_0\}})$, which shows that setting $s_k=t_{k-1}\cup\{x_0\}$ finishes the existence proof. For uniqueness, suppose $s_k,t_k\in[s]^k$ both satisfy the conclusion of the lemma. We will first claim that $s_k,t_k$ must contain $x_0$. Indeed, for any nonempty $t\subseteq s$ we have $x_0\notin\dom(\precdot_t)\subseteq X_{\precdot_\emptyset \max_{\precdot_\emptyset}t}\subseteq X_{\precdot_\emptyset x_0}$, so the desired inclusion cannot hold otherwise. Write $s_k'=s_k\setminus\{x_0\}$ and $t_k'=t_k\setminus \{x_0\}$. Note that we have $s\setminus\{x_0\}\subseteq \dom(\precdot_{x_0})$, so we have $s\setminus s_k\subseteq \dom(\precdot^{x_0}_{s_k'})\subseteq X_{\precdot_\emptyset x_0}$ and $s\setminus t_k\subseteq\dom(\precdot^{x_0}_{t_k'})$. Applying the induction hypothesis proves $s_k'=t_k'$, which completes the proof of uniqueness. 
\end{proof}

\begin{prop}
\label{prop:OVC}
    Suppose $\precdot$ is an $n$-ordering system on a set $X$. Then the collection $\Ff$ of $\precdot$-closed sets has VC-dimension at most $n$.
\end{prop}

\begin{proof}
    Suppose that $s\subseteq X$ has size $n+1$. Apply \Cref{Lemma: inductive maximum}  to get $s'\in[s]^{n-1}$ such that $s\setminus s'\subseteq \dom(\precdot_{s'})$, and write $\{a,b\}=s\setminus s'$ As $\precdot_{s'}$ is a linear order, either $a \precdot_{s'} b$ or $b \precdot_{s'} a$. Assume without loss of generality that the former holds. Then every closed set containing $s',b$ also contains $a$ by \cref{def:closed sets}, so $s$ cannot be shattered.
\end{proof}

The following is a main ingredient in all that follows.

\begin{cor} \label{cor:closures imply bounded VC dimension}
    Suppose $\lambda\leq\theta \leq \kappa$ are infinite cardinals and $n<\omega$ satisfy that there exists an $(n+1)$-ordering system $\precdot$ on $\kappa^{+n}$ with $(\lambda,\theta)$-closures, then the family $\Ff$ of $\precdot$-closed sets of size less than $\theta$ has VC-dimension at most $n+1$. In particular, it witnesses $\VCcof(\lambda,\theta,\kappa^{+n})\leq n+1$. From 
    \cref{lower1} and \cref{monotonicity} it follows that $\VCcof(\lambda,\theta,\kappa^{+n})= n+1$. 
\end{cor}

\begin{proof}
    Fix such an ordering system $\precdot$. By \cref{prop:OVC}, the collection of $\precdot$-closed sets has VC-dimension at most $n+1$. By \cref{monotonicity} so does $\Ff$, and by having $(\lambda,\theta)$-closures this family is $\lambda$-cofinal.
\end{proof}

In \cite{BBNKS}, Bays, Simon and the first two authors use this idea in order to prove the following.
\begin{fct} [\cite{BBNKS},Theorem 3.8] \label{fac:finite closures on omega1}
\label{unctbl}
    There is a 2-ordering system on $\omega_1$ of order type bounded by $(\omega_1, \omega)$ (see \cref{def:bound on order type}) with finite closures, so in particular, $\VCcof(\omega,\omega,\omega_1)=2$.
\end{fct}

\section{Closed initial segments}

Recall from \cref{def:closed sets}: if $0<n$ and $\precdot$ is an $n$-ordering system on $X$ then a set $S \subseteq X$ is $\precdot$-closed if for every $s\subseteq S$ of size $|s|=n-1$ and every $a,b\in \dom(\mathord{\precdot_s})$, if $b\in S$ and $a \precdot_s b$ then $a\in S$. Given $s \subseteq X$ of size $n-1$ and $b \in X$, call the set $s \cup \{b\} \cup \set{a}{a \precdot_s b}$ the \emph{closed $n$-initial segment for $\precdot$ determined by $s \cup \{b\}$}. Note that when $n=1$, the closed 1-initial segment determined by $b$ is precisely the closed initial segment $X_{{\preceqdot_\emptyset}  b} = \set{x\in X}{x\preceqdot_\emptyset b}$ (see \cref{exa:n=1} as well). 

If $S$ is finite of size at least $n$, then $S$ is contained in a closed $n$-initial segment: let $\beta_n = \max_{\precdot_\emptyset} S$, $\beta_{n-1}= \max_{\precdot_{\{\beta_n\}}} S \setminus \{\beta_n\}$, \ldots, $\beta_1 = \max_{\precdot_{\set{\beta_i}{1<i\leq n}}} S \setminus \set{\beta_i}{1<i\leq n}$, then $S$ is contained in the closed $n$-initial segment determined by $\beta_n,\ldots,\beta_1$. Indeed, this follows from \cref{rem:domain of <<s union a}. Moreover, if $S$ is $\precdot$-closed, then the procedure above shows that $S$ itself is a closed $n$-initial segment.

Consider \cref{fac:finite closures on omega1}: the 2-ordering system there on $\omega_1$ has order type bounded by $(\omega_1,\omega)$. This means in particular that closed 2-initial segments are finite. On the other hand, by the previous paragraph all $\precdot$-closed finite sets of size at least 2 are closed 2-initial segments. This naturally raises the question:  can we improve \cref{fac:finite closures on omega1} so that closed 2-initial segments are $\precdot$-closed implying that the two notions coincide (and in particular the 2-ordering system will have finite closures)? This section is dedicated to answering this question by finding a 2-ordering system on $\omega_1$ for which closed 2-initial segments are $\precdot$-closed. Together with \cref{prop:OVC} this answers positively the first part of \cite[Question 4.7]{BBNKS}.


\begin{thm}
\label{Theorem: closed initial segments}
    There is a 2-ordering system $\precdot$ on $\omega_1$ of order type bounded by $(\omega_1,\omega)$ in which all closed 2-initial segments are $\precdot$-closed. In particular, it has finite closures. 
\end{thm}

\begin{proof}
    Given an ordering system $\precdot$ on a set $X$ and $A\subseteq X$, $x\in X$, write $A\vdash x$ if any $\precdot$-closed set $S$ which contains $A$ must contain $x$. 
    
    
    By induction on $\alpha<\omega_1$, we will recursively construct a sequence of 2-ordering systems  $\precdot^\alpha$ on $\alpha +1$ for $\alpha < \omega_1$ such that the following hold:
    \begin{enumerate}
        \item \label{itm:usual ordering} $\precdot^\alpha_\emptyset$ is the usual ordering on $\alpha+1$.
        \item \label{itm:increasing} $\precdot^\alpha_{\{\beta\}} = \precdot^{\alpha'}_{\{\beta\}}$ whenever $\beta \leq \alpha < \alpha'$.
        \item \label{itm:order type} $\precdot^\alpha_{\{\beta\}}$ is of order type $|\beta|$ for all $\beta\leq \alpha$.
        \item \label{IC} All closed 2-initial segments of $\precdot^\alpha$ are $\precdot^\alpha$-closed.
        \item \label{dash} Given a finite and $\precdot^\alpha$-closed set $S\subseteq \alpha$ and $\beta,\gamma\in \alpha\setminus S$ distinct, either $S \cup \{\gamma\}\not\vdash\beta$ or $S\cup \{\beta\}\not\vdash\gamma$. 
    \end{enumerate}

    If we succeed, then define $\precdot$ on $\omega_1$ by setting $\precdot_\emptyset$ as the usual ordering on $\omega_1$ and $\mathord{\precdot}_{\{\alpha\}} = \mathord{\precdot}_{\{\alpha\}}^\alpha$ for $\alpha<\omega_1$ to conclude: (\ref{IC}) tells us that all closed 2-initial segments are $\precdot$-closed.
    
    In the construction, we only explain how to define $\precdot^\alpha_{\{\alpha\}}$ since (\ref{itm:usual ordering}) and (\ref{itm:increasing}) determine the rest. Thus, to simplify notation, we will write $\precdot_{\{\alpha\}}$ instead of $\precdot^\alpha_{\{\alpha\}}$.
    
    Define $\precdot_{\{0\}}$ to be the empty ordering, then clearly it satisfies this. 
    
    We start with the successor step. Assuming we have defined $\precdot^\beta$ for all $\beta \leq \alpha<\omega_1$ such that the above hold, define $\precdot_{\{\alpha+1\}}=\{(\alpha,\beta):\beta<\alpha\}\cup \mathord{\precdot}_{\{\alpha\}}$ (i.e., leave the ordering $\precdot_{\{\alpha\}}$ on $\alpha$ unchanged and add $\alpha$ as the first element). We claim that our assumptions still hold. 
    
    (\ref{itm:usual ordering}) is clear.

    (\ref{itm:order type}): We have to show that the order type of $\prec_{\{\alpha+1\}}$ is $|\alpha+1|$. This follows directly by the construction and the induction hypothesis.

    (\ref{IC}): Suppose that $D$ is a closed 2-initial segment of $\precdot^{\alpha+1}$, i.e., has the form $\set{x<\alpha+2}{x \precdot_{\{\beta\}} \gamma} \cup \{\beta,\gamma\}$ for some $\gamma<\beta<\alpha+2$. If $\beta<\alpha+1$, then $D$ is $\precdot^{\alpha}$-closed by induction, so also $\precdot^{\alpha+1}$-closed. If $\gamma = \alpha$ (i.e., the minimal element in $\precdot_{\{\alpha+1\}}$) then $D = \{\beta, \gamma\}$ so it is trivially $\precdot^{\alpha+1}$-closed, so we may assume that $\beta=\alpha+1,\gamma \neq \alpha$. From this and the definition of $\precdot_{\{\alpha+1\}}$, it follows that $D \cap (\alpha+1)$ is the closed 2-initial segment of $\precdot^\alpha$ determined by $\alpha,\gamma$, so it is $\precdot^\alpha$-closed by induction.
    
    We are given $\gamma_2,\delta <\gamma_1\leq \alpha+1$ such that $\gamma_2,\gamma_1\in D$ and $\delta\precdot_{\{\gamma_1\}} \gamma_2$ and we have to show that $\delta \in D$. If $\gamma_1 = \alpha+1$, this follows directly from the definition of $D$, so we may assume  $\gamma_1\leq \alpha$. In this case $\gamma_2,\gamma_1 \in D \cap (\alpha+1)$, so $\delta \in D \cap (\alpha+1)$ since it is $\precdot^\alpha$-closed, as required.

    
    

    
    (\ref{dash}): Suppose $S\subseteq\alpha+1$ is finite and $\precdot^{\alpha+1}$-closed, and $\beta,\gamma\in\alpha+1\setminus S$ are distinct and $\gamma<\beta$.  
    
    
    If $S=\emptyset$ then (\ref{dash}) holds since singletons are closed. If $S\cup\{\beta\}\subseteq\alpha$ then (\ref{dash}) follows from the induction hypothesis. Otherwise, we either have $\alpha\in S$ or $\beta=\alpha$.
    \begin{itemize}
        \item If $\alpha\in S$ and $\beta,\gamma\in\alpha+1\setminus S$, we must have $\beta,\gamma<\alpha$. Suppose without loss of generality that  $\beta\precdot_{\{\alpha\}}\gamma$. Note that as $S$ is $\precdot^{\alpha+1}$-closed and $\beta \notin S$, $S \subseteq \set{x<\alpha}{x\precdot_{\{\alpha\}} \beta} \cup \{\alpha\}$. Thus, the set $\set{x<\alpha}{x\precdot_{\{\alpha\}} \beta} \cup \{\alpha,\beta\}$ is a closed 2-initial segment of $\prec^{\alpha}$ not containing $\gamma$. By (\ref{IC}) applied to $\precdot^{\alpha}$, this set witnesses that $S\cup\{\beta\}\not\vdash \gamma$.
        
        \item Otherwise we have $\beta=\alpha$. Take some initial segment $D$ of $\precdot_{\{\alpha\}}$ which contains $\gamma$ and $S$. Note that $D \cup \{\alpha\}$ is a closed 2-initial segment of $\precdot^\alpha$ so $\precdot^\alpha$-closed by induction. It follows that $D$ is also $\precdot^\alpha$-closed (clearly if $\gamma_1,\gamma_2 \in D$ and $\delta \precdot_{\{\gamma_1\}} \gamma_2$ then $\delta <\alpha$). Thus, $D$ witnesses that $S\cup \{\gamma\}\not\vdash \beta$.
    \end{itemize}
    
    We are left with the limit stage of the construction. Suppose $\delta<\omega_1$ is a limit ordinal and we constructed $\precdot^\alpha$ for all $\alpha<\delta$. As above, we only need to define $\precdot_{\{\delta\}}$ on $\delta+1$. 
    
    Fix a bijection $f:\omega\rightarrow\delta$. Let $\mathord{\precdot^*}$ be a 2-ordering system on $\delta$ determined by $\precdot^*_\emptyset$ being the usual order on $\delta$ and for each $\alpha < \delta$, $\mathord{\precdot^*_{\{\alpha\}}} = \mathord{\precdot^{\alpha}_{\{\alpha\}}}$. Note that (\ref{itm:usual ordering})--(\ref{dash}) hold for $\precdot^*$. Define an increasing union of finite $\precdot^*$-closed sets $(S_n)_{n<\omega}$ recursively as follows. Take $S_0=\emptyset$. Suppose we have defined the $\precdot^*$-closed sets $(S_k)_{k\leq n}$ such that each $S_k$ contains $f(0),\ldots,f(k-1)$ and $S_j$ for $j<k$. Let $\alpha < \delta$ be such that $\alpha+1$ contains $S_n$. By (\ref{IC}) (and (\ref{itm:order type}), i.e., the fact that closed 2-initial segments are finite), there is some finite $S_{n+1} \subseteq \alpha+1$ which is $\precdot^\alpha$-closed (so also $\precdot^*$-closed) and contains $S_n$ and $f(n)$. Note that we have $\bigcup_{n<\omega}S_n=\delta$. 

    Fix some $n<\omega$. Suppose that $m = |S_{n+1} \setminus S_n|$. By induction on $i\leq m$ define a strictly increasing sequence of injective functions (enumerations) $f_i:i\to S_{n+1} \setminus S_n$ such that for all $i<m$, $S_n \cup \im(f_i)$ is $\precdot^*$-closed (where $\im(f_i)$ is the image of $f_i$).
    
    For $i=0$, this is possible since $S_n$ is $\precdot^*$-closed. Given $f_i$ for $i<m$, define the following binary relation on $S_{n+1}\setminus \im(f_i)$: $x \preceq y$ if and only if $S_n \cup \im(f_i) \cup \{y\} \vdash x$. Note that $\precdot$ is antisymmetric: if $x\preceq y$ and $y \preceq x$ then $x=y$ by (\ref{dash}) (applied to the $\precdot^*$-closed set $S_n \cup \im(f_i)$). It is also clearly transitive, so it is a partial order on a finite set and thus has a minimal element $x_0$, and let $f_{i+1}(i)=x_0$. Let us show that $S_n \cup \im(f_{i+1}) = S_n \cup \im(f_i) \cup {\{x_0\}}$ is $\prec^*$-closed. Indeed, define the closure $\cl(S_n \cup \im(f_{i+1}))$ as the intersection of all $\precdot^*$-closed sets containing it (note that $S_{n+1}$ is closed, so this closure is finite and contained in $S_{n+1}$). Clearly $\cl(S_n \cup \im(f_{i+1}))$ is $\precdot^*$-closed and contains $S_n \cup \im(f_{i+1})$. If $ y \in \cl(S_n \cup \im(f_{i+1})) \setminus (S_n \cup \im(f_{i+1}))$ then $y \preceq x_0$ and $y \neq x_0$, contradicting the minimality of $x_0$. Thus, $\cl(S_n \cup \im(f_{i+1})) = S_n \cup \im(f_{i+1})$ and the latter is $\precdot^*$-closed as required. 

    Being done with the inductive construction, let $f^n = \bigcup_{i\leq m} f_i$; it is an enumeration of $S_{n+1} \setminus S_n$. For $x <\delta$, let $n_x = \max \set{n<\omega}{x\notin S_n}$.

    Finally define $\precdot_{\{\delta\}}$ by $x\precdot_{\{\delta\}} y$ if and only if $n_x < n_y$ or $n=n_x = n_y$ and $(f^{n})^{-1}(x) < (f^{n})^{-1}(y)$ (i.e., $x$ appears before $y$ in the enumeration given by $f^n$). Note that any initial segment of $\precdot_{\{\delta\}}$ is $\precdot^*$-closed.

    Let us check that (\ref{itm:usual ordering})--(\ref{dash}) hold. (\ref{itm:usual ordering}), (\ref{itm:increasing}) are clear and (\ref{itm:order type}) holds since every initial segment of the linear order $\precdot_{\{\delta\}}$ is finite (so that the order type of $\precdot_{\{\delta\}}$ is $\omega$). 
    
    We show (\ref{IC}). Suppose that $D = \set{x<\delta}{x\precdot_{\{\alpha\}} \beta} \cup \{\alpha,\beta\}$ for some $\beta<\alpha\leq \delta$ is a closed 2-initial segment of $\precdot^\delta$. If $\alpha<\delta$, $D$ is $\precdot^*$-closed (so also $\precdot^\delta$-closed) by the induction hypothesis, so suppose that $\alpha = \delta$. We are given $\gamma_2,\varepsilon<\gamma_1 \leq \delta$ such that $\gamma_1,\gamma_2 \in D$ and $\varepsilon\precdot_{\{\gamma_1\}} \gamma_2$ and we have to show that $\varepsilon \in D$. If $\gamma_1 = \delta$, this follows directly from the definition of $D$, so we may assume that $\gamma_1 < \delta$. But since $D \cap \delta$ is an initial segment of $\precdot_{\{\delta\}}$, it is $\precdot^*$-closed, so $\varepsilon \in D$ as required.

    Finally we show (\ref{dash}). Suppose that $S \subseteq \delta$ is finite and $\precdot^\delta$-closed, and $\alpha\neq \beta \in \delta \setminus S$. Then, $S \cup \{\alpha,\beta\} \subseteq \gamma < \delta$, and hence this follows from the induction hypothesis applied to $\precdot^\gamma$.

    \end{proof}

\section{Large cofinality and uncountable cardinals}

Recall \cref{def:VC cof}:  Given cardinals $\lambda\leq\theta\leq\kappa$, $\VCcof(\lambda,\theta,\kappa)$ is the least number $n\in\omega\cup\{\infty\}$ such that there is a $\lambda$-cofinal family $\Ff\subseteq[\kappa]^{<\theta}$ of VC-dimension $n$. In this section we study the case of an uncountable bound $\theta$ on the size of the members (i.e. the second parameter of $\VCcof$). 

We prove that when $\lambda\leq\theta\leq\kappa$ are infinite cardinals, $n<\omega$ and $\theta$ uncountable then the behavior of $\VCcof(\lambda,\theta,\kappa)$ depends on whether $\theta$ is regular or singular. 

In the regular case we get an optimal result: $\VCcof(\lambda,\theta,\kappa)=n+1$ if and only if $\kappa=\theta^{+n}$. In particular, $\lambda$ is not restricted. 

In the singular case, we show that in VC-classes the cofinality parameter $\lambda$ (the first parameter of $\VCcof$) in a VC-class can be at most $\cof(\theta)$. It is not clear what can be said about $\VCcof(\lambda,\theta,\kappa)$ when $\lambda \leq \cof(\theta)$, and see \cref{que:singular cardinals general question} for a precise formulation. 

However, we do show in \cref{prop:measurable} that relative to the existence of a measurable\footnote{See \cite{Shoenfield} for a comprehensive review of measurable cardinals.} cardinal it is consistent that there is an uncountable cardinal $\theta$ of cofinality $\omega$ such that for every $n<\omega$, $\VCcof(\omega,\theta,\theta^{+n})=n+1$.

\begin{nota}
    Given a $k$-ordering system $\precdot$ on a set $X$ and $s\in[X]^{<n}$, write $\dom(\mathord{\precdot_s})$ for the domain of the ordering $\precdot_s$ (so a subset of $X$).
\end{nota}


    
    
We now have a useful lemma, which will be used at its full generality in the next section.

\begin{lem}
\label{exists}
    Suppose $\theta$ is an infinite cardinal, $n<\omega$ and $X$ is a set satisfying $|X|\leq \theta^{+n}$. Then there exists an $(n+1)$-ordering system $\precdot$ on $X$ such that $\mathord{\precdot}_\emptyset=\mathord{<}_X$ and $(\theta^{+n},\ldots,\theta)$ bounds the order type of $\precdot$ (see \cref{def:bound on order type}).
\end{lem}

\begin{proof}
    By induction on $n$. We first set $\mathord{\precdot}_\emptyset$ to be some linear order on $X$ of order type $|X|$. If $n=0$ then we are done. Otherwise, since the set $X_{\precdot_{\emptyset}x_0}$ has cardinality at most $\theta^{+(n-1)}$ for every $x_0\in X$, we can apply the induction hypothesis to get an $n$-ordering systems $\precdot^{x_0}$ on each of these sets of order type bounded by $(\theta^{+(n-1)},...,\theta)$, and so for $s\in [X]^{<n}\setminus\{\emptyset\}$ we define $\mathord{\precdot}_{s}=\mathord{\precdot}^{x_0}_{s\setminus \{x_0\}}$, where $x_0=\max_{\precdot_\emptyset}s$.
\end{proof}

\begin{prop}\label{regular}
    Fix an uncountable regular cardinal $\theta$ and a natural number $n<\omega$, then any $(n+1)$-ordering system $\precdot$ on $\theta^{+n}$ with $(\theta^{+n},\ldots,\theta)$ bounding its order type has $\theta$-closures (see \cref{def:theta-closures}). In particular, we have $\VCcof(\theta,\theta,\theta^{+n})=n+1$.    
\end{prop}

\begin{proof}
    Fix such an ordering system $\precdot$ and fix $A\in[\theta^{+n}]^{<\theta}$. Define $A=A_0$, and given that we have defined $A_k$ for some $k<\omega$, define: 
    
    $$A_{k+1}=A_k\cup\bigcup\left\{\{x\precdot_s \alpha\}:\alpha\in \dom\left(\precdot_s\right)\cap A_k,s\in [A_k]^n\right\}.$$
    
    By definition we have $|A_0|<\theta$, and by induction, the order type bound assumption and the regularity of $\theta$, it follows that $|A_{k+1}|<\theta$. Take $A_\omega=\bigcup_{k<\omega}A_k$ hence $|A_\omega|<\theta$ by the regularity and uncountability of $\theta$. Clearly it is $\precdot$-closed. 
    Hence $\precdot$ has $\theta$-closures. For the second part, apply \cref{cor:closures imply bounded VC dimension,exists} to the first part. 
\end{proof}

We move on to the case of singular cardinals.

\begin{lem}
\label{restrict}
    Suppose $\theta$ is singular of cofinality $\lambda$ and $\Ff\subseteq[\theta]^{<\theta}$ is $\lambda^+$-cofinal. Then there is some infinite cardinal $\kappa<\theta$ and $\Ff_0\subseteq\Ff$ with $\Ff_0\subseteq[\theta]^{<\kappa}$ that is $\lambda^+$-cofinal.
\end{lem}

\begin{proof}
    Assume otherwise. Let $(\kappa_\alpha:\alpha<\lambda)$ be a strictly increasing sequence of cardinals with $\bigcup_{\alpha<\lambda}\kappa_\alpha=\theta$. Then there are $A_\alpha\in[\theta]^{<\lambda^+}$ for $\alpha<\lambda$ such that the minimal cardinality of a set $B_\alpha\in\Ff$ with $A_\alpha\subseteq B_\alpha$ is greater than $\kappa_\alpha$. Let $A=\bigcup_{\alpha<\lambda}A_\alpha$, so it is of size $\leq \lambda$ as well, hence there is some $B\in\Ff$ with $A\subseteq B$. Since $(\kappa_\alpha:\alpha<\lambda)$ is cofinal in $\theta$ and $|B|<\theta$, there is some $\alpha<\lambda$ with $|B|<\kappa_\alpha$. But $A_\alpha\subseteq A\subseteq B_\alpha$, a contradiction to our choice of $A_\alpha$.
\end{proof}

\begin{cor}
\label{singular}
    Suppose $\theta$ is singular of cofinality $\lambda$. Any $\lambda^+$-cofinal $\Ff\subseteq[\theta]^{<\theta}$ has infinite VC-dimension.
\end{cor}

\begin{proof}
    Let $\Ff$ be as above. Apply \cref{restrict} to get some $\kappa$ and $\Ff_0\subseteq\Ff$ as there. Since $\theta$ is singular, it is a limit cardinal, and because $\kappa<\theta$ we get $\kappa^{+n}<\theta$ for every $n<\omega$. Let $\Ff_0' = \set{s \cap \kappa^{+n}}{s\in \Ff_0}$. Assuming towards a contradiction that $\VC(\Ff)=n<\omega$, and in particular $\VC(\Ff_0')\leq n$ by \cref{monotonicity}. This gives us $\VCcof(\omega,\kappa,\kappa^{+n})<n+1$, a contradiction to \cref{lower1}.
\end{proof}

Therefore, for a singular $\theta$, we have a critical point $\lambda=\cof(\theta)$ which would be of interest.

We now demonstrate that consistently (assuming the consistency of a measurable cardinal), at least for some cardinal this is indeed the critical value.

Recall the following celebrated result of Prikry, originating in \cite{Prikry}.

\begin{fct}{\cite[Theorem 1.10]{Gitik}}
\label{Prikry}
    Assume the existence of a measurable cardinal $\theta$. Then there exists a forcing notion $\mathbb{P}$ such that the following hold:
    \begin{enumerate}
        \item $\mathbb{P}$ preserves cardinals.
        \item $\Vdash_{\mathbb{P}}\theta\text{ is uncountable of cofinality $\omega$}$.
    \end{enumerate}
\end{fct}

\begin{prop} \label{prop:measurable}
    Assume the consistency of a measurable cardinal. Then it is consistent that there is an uncountable cardinal $\theta$ of cofinality $\omega$ such that for every $n<\omega$ we have $\VCcof(\omega,\theta,\theta^{+n})=n+1$.
\end{prop}

\begin{proof}
    Let $\theta$ be measurable and let $\mathbb{P}$ be the forcing notion of \cref{Prikry}. Since $\theta$ is measurable, it is regular, so by \cref{regular} for every $n<\omega$ there is some $\Ff_n$ witnessing $\VCcof(\theta,\theta,\theta^{+n})=n+1$. Let $G\subseteq\mathbb{P}$ be a generic filter and fix $n<\omega$. By \cref{Prikry} the cardinal $\theta$ is singular and of cofinality $\omega$ in $V[G]$, and the $n$'th successor of $\theta$ in $V[G]$ is its $n$'th successor in $V$, so it remains to show that $\Ff_n$ is $\omega$-cofinal and $\VC(\Ff_n)=n+1$. For the first of these, forcing cannot add a new finite subset to an existing set, so by the $\theta$-cofinality of $\Ff_n$ in $V$ we get $\omega$-cofinality in $V[G]$. The latter follows from absoluteness, hence we are done. 
\end{proof}

Despite of this consistency result, the general case of $\lambda$ in the range $\omega\leq \lambda\leq \cof(\theta)$ remains to be considered.

\begin{que} \label{que:singular cardinals general question}
    Is it true that for every singular cardinal $\theta$ and $n<\omega$ we have $\VCcof(\cof(\theta),\theta,\theta^{+n})=n+1$? Is it consistent?
\end{que}

\section{Cohen-generic ordering systems on $\omega_n$}
\label{Section: forcing}

Given the second result mentioned in \cref{unctbl}, we would like to see whether there exists a cofinal family $\Ff\subseteq[\omega_n]^{<\omega}$ of VC-dimension $n+1$. The goal of this section is to establish its consistency. 

So far, to construct $(n+1)$-ordering systems, we constructed it inductively, first choosing $\precdot_\emptyset$ and then defining $\precdot_s$ for nonempty $s$ by an inductive construction. Here we start with an $n$-ordering system and use forcing to construct $\precdot_s$ for $s$ of size $n$. The domain of these new orders can be already defined without forcing, and we want to define these in a uniform way that allows us to argue about these new orders as well as old orders (where the domain is already defined). This motivates the definition of a ``pre-domain'' which we now explain.

\begin{defn}
    Given an $n$-ordering system $\precdot$ on a set $X$ and $s\in[X]^k$ for $0<k\leq n$, write $\Min_{\precdot} s$ for the unique element of the set $s\setminus s_{k-1}$, where $s_{k-1}$ is the subset given by \Cref{Lemma: inductive maximum}. Namely, it is the unique $y\in s$ such that $y\in\dom(\precdot_{s\setminus\{y\}})$.
\end{defn}


\begin{nota}
\label{Notation: pre-domain}
    Given an $n$-ordering system $\precdot$ on a set $X$ and given $s\in[X]^{\leq n}$, if $s\neq\emptyset$, we write $\predom(\precdot_s):=\dom(\precdot_{s'})_{\precdot_{s'}\mathord{\Min_{\precdot}} s} = \set{x\in \dom(\precdot_{s'})}{x \precdot_{s'} \Min_{\precdot} s}$, 
    and if $s=\emptyset$, we write $\predom(\precdot_\emptyset)=X$.
\end{nota}

\begin{rem}
\label{rem: dom is predom}
    By applying \Cref{rem:domain of <<s union a,Lemma: inductive maximum}, if $s\in[X]^k$ for $k<n$ then we have $\predom(\precdot_s)=\dom(\precdot_s)$.
\end{rem}

We have the following lemma.

\begin{lem}
\label{Lemma: two predomains}
    Suppose $\precdot$ is an $n$-ordering system on a set $X$ and $x_0\in X$. Then for any $t\in [\dom(\precdot_{\{x_0\}})]^{<n}$ we have $\predom(\precdot^{x_0}_t)=\predom(\precdot_{t\cup\{x_0\}})$.
\end{lem}

\begin{proof}
     We may assume $t\neq \emptyset$. By the definition of $\predom$, the lemma would follow once we prove that $\Min_{\precdot}(t\cup\{x_0\})=\Min_{\precdot^{x_0}}t$. Writing $m$ for the latter, we get:
    $$m\in\dom(\precdot^{x_0}_{t\setminus\{m\}})=\dom(\precdot_{(t\cup\{x_0\})\setminus\{m\}}),$$
    which is the definition of $\Min_{\precdot^{x_0}}t$.
\end{proof}

We immediately get the following.

\begin{cor}
\label{Corollary: bottom-up}
    Suppose $\precdot$ is an $n$-ordering system on a set $X$, and for every $s\in[X]^n$ we are given a well-order $<_s$ on the set $\predom(\precdot_s)$. Then the map $\precdot':[X]^{\leq n}\rightarrow \mathcal{P}(X^2)$ which extends $\precdot$ and assigns $\mathord{\precdot}'(s)=\mathord{<}_s$ to any $s\in[X]^n$ is an $(n+1)$-ordering system on $X$.
\end{cor}

\begin{proof}
    By induction on $n$, with the case of $n=1$ being the definition of a 2-ordering system. Suppose $n>1$ and fix $x_0\in X$. In order to show that $\precdot'$ is an $(n+1)$-ordering system on $X$, we must show that $(\precdot')^{x_0}$ is an $n$-ordering system on $\dom(\precdot_{x_0})$. 
    
    Now, the assignment $(\precdot^{x_0})':[\dom(\precdot_{\{x_0\}})]^{<n-1}\rightarrow \mathcal{P}([\dom(\precdot^{x_0})]^2)$ extending $\precdot^{x_0}$ and assigning to each $t\in[\dom(\precdot_{\{x_0\}})]^{n-1}$ the orderings $\mathord{<'}_t:=\mathord{<}_{t\cup\{x_0\}}$ on each $\predom(\precdot_{t\cup \{x_0\}})$ (which is equal to $\predom(\precdot^{x_0}_t)$ by \Cref{Lemma: two predomains}) is an $n$-ordering system on $\dom(\precdot_{\{x_0\}})$ by the induction hypothesis, which is exactly the assignment $(\precdot')^{x_0}$. Thus the proof is complete.
\end{proof}

We note that if there is a uniform bound $\alpha$ on the order type of the orderings $<_s$, and $\precdot$ has order type bounded by $(\alpha_{n-1},...,\alpha_0)$ (see \Cref{def:bound on order type}), then $\precdot'$ as constructed in \Cref{Corollary: bottom-up} has order type bounded by $(\alpha_{n-1},...,\alpha_0,\alpha)$.

\begin{defn}
\label{Definition: nice}
    Let $(X,<^X)$ be a well-ordered set. For an ordinal $\alpha\leq \omega$, write $\alpha^{<^X}$ for the first $\alpha$ elements of $<^X$ (or all of $X$ if $\otp(<^X)<\alpha$). An $n$-ordering system $\precdot$ on $X$ is \emph{nice} if it satisfies the following:

    \begin{enumerate}
        \item We have $\mathord{\precdot_\emptyset}=\mathord{<^X}$.
        \item Writing $\alpha=\min\{|\dom(\precdot_s)|,\omega\}$ we have $(\alpha^{<^X},<^X)=(\alpha^{\precdot_s},\precdot_s)$ (or equivalently, $\alpha^{<^X}$ is contained in $\dom(\precdot_s)$ as an initial segment and $<^X\restriction \alpha^{<^X}=\mathord{\precdot_s}\restriction\alpha^{<^X}$).
    \end{enumerate}    
\end{defn}

\begin{rem}
\label{Remark: nice reduction}
    Note that for any nice $n$-ordering system on $(X,<^X)$ and $x_0\in X$, writing 
    $\alpha=\min\{|\dom(\mathord{\precdot}_{\{x_0\}})|,\omega\}$, since $\alpha^{<^X}=\alpha^{\precdot_{\{x_0\}}}$ and the two orderings $<^X,\precdot_{\{x_0\}}$ agree on this set, we have that $\precdot^{x_0}$ is a nice $(n-1)$-ordering system on $(X_{<^X x_0},\precdot_{\{x_0\}})$.
\end{rem}

The following proof is an elaboration of \Cref{exists}.

\begin{lem}
\label{exists1}
    Suppose $(X,<^X)$ is an infinite well-ordered set satisfying $\otp(\mathord{<^X})\leq \omega_n$ for some $n>0$. Then there exists a nice $n$-ordering system $\precdot$ on $X$ with $(\omega_n,\ldots,\omega_1)$ bounding the order type of $\precdot$ (see \cref{def:bound on order type}).    
\end{lem}

\begin{proof}
     By induction on $n$. We set $\mathord{\precdot_\emptyset}=\mathord{<^X}$. If $n=1$ we are done. Otherwise, since the initial segment $X_{<^X x_0}$ has cardinality at most $\aleph_{n-1}\geq \aleph_1$ for any $x_0\in X$, we fix a well-ordering $<_{x_0}$ of order type at most $\omega_{n-1}$ on each $X_{<^X x_0}$, such that $\omega^{<^X}\cap X_{<^{X} x_0}$ is an initial segment of $(X_{<^{X} x_0},<_{x_0})$ with both $<^X,<_{x_0}$ agreeing on $\omega^{<^X}\cap X_{<^{X} x_0}$. Thus, we have that $\omega^{<_{x_0}}$ is an initial segment of $\omega^{<^X}$. Applying the induction hypothesis to $(X_{<^X x_0},<_{x_0})$ for every $x_0$ gives us nice $(n-1)$-ordering systems $\precdot^{x_0}$ on each of these sets of order type bounded by $(\omega_{n-1},...,\omega_1)$, and so for $s\in [X]^{<n}\setminus\{\emptyset\}$ we define $\mathord{\precdot_s}=\mathord{\precdot^{x_0}_{s\setminus\{x_0\}}}$, where $x_0=\max_{\precdot_\emptyset}s$.
\end{proof}

We will start with a lemma.

\begin{lem}
\label{lem: finiteness condition}
    Suppose $\precdot$ is a nice $n$-ordering system on a well-ordered set $(X,<^X)$. For $s\in[X]^{\leq n}\setminus\{\emptyset\}$ we have that $\predom(\precdot_s)$ is finite if and only if  $s\cap\omega^{<^X}\neq\emptyset$. In this case we have:
    $$\predom(\precdot_s)=\{x\in X:x<^X\min(s\cap\omega^{<^X})\}.$$
\end{lem}

\begin{proof} 
    Fix $s\in[X]^{\leq n}\setminus\{\emptyset\}$. The case of $|s|=1$ is immediate, so assume $|s|>1$ and write $x_0=\max_{\precdot_\emptyset} s$. We also write $k=\min\{|\dom(\precdot_{\{x_0\}})|,\omega\}$.

    We will prove the lemma by induction on $n$, where the case $n=1$ follows from the case $|s|=1$.

    For $n>1$ we have $\predom(\precdot^{x_0}_{s\setminus\{x_0\}})=\predom(\precdot_s)$ by \Cref{Lemma: two predomains}. By the induction hypothesis and since $k^{<^X}=k^{\precdot^{x_0}}$ by niceness, we first of all get that if the former is finite then $\predom(\precdot_s)=\left\{x\in X: x<^X \min((s\setminus\{x_0\})\cap \omega^{<^X})\right\}$ (where here we use that niceness gives $\mathord{<}^X=\mathord{\precdot}_{x_0}$ on $\omega^{<^X}$). We also get that $\predom(\precdot_s)$ is finite if and only if $(s\setminus\{x_0\})\cap \omega^{<^X}\neq\emptyset$. We now finish both parts by noting that if $\predom(\precdot_s)$ is finite then since $|s|\geq 2$, any $y\in s\setminus\{x_0\}$ must satisfy $y\precdot x_0$ by maximality, implying both that $\min((s\setminus\{x_0\}\cap\omega^{<^X})=\min(s\cap \omega^{<^X})$ and that if $s\cap\omega^{<^X}\neq\emptyset$ then $s\setminus\{x_0\}\cap\omega^{<^X}\neq\emptyset$. This finishes the proofs of both parts of the lemma.
\end{proof}

\begin{lem}
\label{Corollary: contains omega}
    In the context of \Cref{lem: finiteness condition}, suppose $\predom(\precdot_s)$ is infinite. Then it contains $\omega^{<^X}$.
\end{lem}

\begin{proof}
    Write $y=\Min_{\precdot}s$ and $s'=s\setminus\{y\}$. By definition we have that $\predom(\precdot_s)$ is an initial segment of the ordering $\precdot_{s'}$ on the set $\dom(\precdot_{s'})$, which contains $\omega^{<^X}$ as an initial segment of order type $\omega$ by niceness. Since any infinite initial segment of a well-order contains its first $\omega$ elements, the result follows.
\end{proof}

Throughout the rest of this subsection, fix some $2\leq n<\omega$, and a nice $n$-ordering system $\precdot$ on $(\omega_n,<)$ (see \Cref{Definition: nice}) whose order type is bounded by $(\omega_n,...,\omega_1)$ (\cref{def:bound on order type}), where $<$ is the usual ordering on $\omega_n$. Note that in particular we have $\omega=\omega^{<}$.

\begin{nota}
    Define the following (where the second equality follows from \Cref{lem: finiteness condition}): 
    $$\Inf_{\precdot}=\{s\in[\omega_n]^{n}:|\predom(\precdot_s)|\geq \aleph_0\}=\{s\in[\omega_n]^n:s\cap\omega=\emptyset \}.$$
\end{nota}

We will now prove another lemma.

\begin{lem}
\label{countable}
    For any $s\in[\omega_n]^n$ we have $|\predom(\precdot_s)|\leq \aleph _0$. In particular, if $s\in\Inf_{\precdot}$ then $|\predom(\precdot_s)|=\aleph_0$.
\end{lem}

\begin{proof}
    Given $s\in[\omega_n]^{n}$ we write $y=\Min_{\precdot}s$ and $s'=s\setminus\{y\}$. By the assumption on the bound of the order type of $\precdot$ we have $\otp(\precdot_{s'})\leq \omega_1$ hence every  initial segment is at most countable.
\end{proof}

 \begin{nota}   
    Given sets $X,Y,Z$, define a partial function $f:X\times Y\rightharpoonup Z$ and some $y\in Y$, write $f_y:X\rightharpoonup Z$ for the function $f_y(x)=f(x,y)$.
    We will often not distinguish between a partial function and its graph. Because we often deal with partial functions from a product of sets, an element of the domain will be a tuple. In order to simplify notation, when writing an element of the graph of a partial function we will separate the elements of the domain from those of the range by a semicolon. That is, if $f:\prod_{i=1}^n X_i\rightharpoonup Y$, $x_i\in X_i$ and $y\in Y$, we will write $(x_1,...,x_n;y)\in f$ to mean $(x_1,...,x_n)\in\dom(f)$ and $f(x_1,...,x_n)=y$.
\end{nota}

\begin{rem}
    This section of the paper uses forcing. We use the Jerusalem forcing convention, under which $p\leq q$ means that $q$ strengthens $p$.
\end{rem}

\begin{defn}
    Write $\mathcal{Q}$ for the set consisting of the finite partial functions $p:\bigcup_{s\in\Inf_{\precdot}}\predom(\precdot_s)\times\{s\}\rightharpoonup\omega$ such that for every $s\in\Inf_{\precdot}$, the partial function $p\restriction \predom(\precdot_s)\times\{s\}$ is injective. We consider it as a forcing notion by equipping it with the partial ordering of inclusion of partial functions.
\end{defn}




By \Cref{countable} we have that $Q$ is isomorphic to the poset finite partial functions $p:\omega\times\omega_n\rightharpoonup\omega$ which are injective in one coordinate, equipped with the extension ordering. Thus, the following Lemma is standard and extends \cite[Lemma VII.5.4]{Kunen}.

\begin{lem}
\label{ccc}
    The forcing $\mathcal{Q}$ satisfies the countable chain condition.
\end{lem}

The following is standard as well, and essentially follows from the same argument as \cite[Lemma VII.5.2]{Kunen}. 

\begin{prop}
\label{Proposition: Cohen package}
    Suppose $G\subseteq\mathcal{Q}$ is a generic filter. Write $g=\bigcup G$. Then the following hold in $V[G]$:
    \begin{enumerate}
        \item $g$ is a function on its domain.
        \item $\dom(g)=\bigcup_{s\in\Inf_{\precdot}}\predom(\precdot_s)\times\{s\}$.
        \item For every $s\in\Inf_{\precdot}$, the function $g\restriction \predom(\precdot_s)\times\{s\}$ is injective.
    \end{enumerate}
\end{prop}


\begin{rem}
    Note that in $V[G]$, $\precdot$ is still a nice $n$-ordering system on $\omega_n=\omega_n^{V[G]}$.
\end{rem}

\begin{defn}
\label{Definition: forced system}
    Suppose $G\subseteq\mathcal{Q}$ is a generic filter. We define $\precdot^{G}$ to be the $(n+1)$-ordering system on $\omega_n$ which is given by applying \Cref{Corollary: bottom-up} to the orderings $<_s$, $s\in[\omega_n]^n$ which are defined as follows (in $V[G]$):
    \begin{itemize}
        \item If $s\notin\Inf_{\precdot}$ then $\mathord{<}_s=\mathord{<^{\Ord}}\restriction \predom(\precdot_s)$ (in this case $\predom(\precdot_s)$ is an initial segment of $\omega$ by \Cref{lem: finiteness condition}).
        \item If $s\in\Inf_{\precdot}$ then for every $x,y\in \predom(\precdot_s)$ write $x<_s y$ if and only if $g_s(x)<g_s(y)$.
    \end{itemize}
\end{defn}

Our goal is to show the following. Its proof will be provided after \Cref{red,const}

\begin{thm}
\label{fc}
    In $V[G]$, the $(n+1)$-ordering system $\precdot^{G}$ has finite closures.
\end{thm}

\begin{nota}
    For $B\subseteq\omega_n$ we write $\Inf_{\precdot}(B)=\Inf_{\precdot}\cap [B]^n$.
\end{nota}

The key observation is the following. 

\begin{prop}
\label{red}
    Suppose $B\in[\omega_n]^{<\omega}$ and $p\in\mathcal{Q}$ satisfy the following:
    \begin{enumerate}
        \item \label{jt1} The set $B\cap \omega$ is an initial segment of $\omega$.
        \item \label{jt2} For all $s\in\Inf_{\precdot}(B)$ we have $\predom(\precdot_s)\cap B=\dom(p_s)$.
        \item \label{jt3} For all $s\in\Inf_{\precdot}(B)$ we have that $\im(p_s)\subseteq\omega$ is an initial segment.
    \end{enumerate}
    Then $p\Vdash B\text{ is $\dot{\precdot^{G}}$-closed}$.
\end{prop}

\begin{proof}
    Let $B,p$ be as above. Let $G\subseteq\mathcal{Q}$ be a generic filter containing $p$ and $\precdot^{G}$ the corresponding $(n+1)$-ordering system as in \Cref{Definition: forced system}. Suppose $t\in[B]^n$, $\alpha,\beta\in\dom(\precdot^G_t)=\predom(\precdot_t)$ and $\alpha\in B$ satisfy $\beta\precdot^G_t\alpha$. Assume first that $t\notin\Inf_{\precdot}$, so $\predom(\precdot_t)\in\omega$ by \Cref{lem: finiteness condition} and the first bullet of \Cref{Definition: forced system}, $\precdot^G_t$ is the usual ordering on this natural number. Thus, by (\ref{jt1}) we have $\beta\in B$. Consider now the case $t\in\Inf_{\precdot}$, meaning $g_t(\beta)<g_t(\alpha)$. Since $\alpha\in \predom(\precdot_t)\cap B$ we have $
    \alpha\in\dom(p_t)$ by (\ref{jt2}), and so (\ref{jt3}) and the injectivity part of \Cref{Proposition: Cohen package} gives us $\beta\in \dom(p_t)$ and hence $\beta\in B$ by (\ref{jt2}). More explicitly, by \cref{Proposition: Cohen package} there is some $q\in G$ such that $(\beta,t) \in \dom(q)$. Additionally, by (\ref{jt2}) there is some $\beta'\in \dom(p_t)$ such that $p_t(\beta')=g_t(\beta)$. Since $G$ is a filter, $p$ and $q$ are compatible, i.e., there is some $r\in G$ containing both. So $p_t(\beta') = g_t(\beta) = r_t(\beta)$ so $\beta = \beta'$ because $r_t$ is injective.
\end{proof}


\begin{prop}[Small edit at end of proof of 3 and change proof of 4]
\label{const}
    Let $A\in[\omega_n]^{<\omega}$ and $p\in\mathcal{Q}$. Then there is some $B\in[\omega_n]^{<\omega}$ and $q\in\mathcal{Q}$ such that the following all hold.
    \begin{enumerate}
        \item \label{J1} $p\leq q$ and $A\subseteq B$.
        \item \label{J2} The set $B\cap\omega$ is an initial segment of $\omega$.
        \item \label{J3} For every $s\in\Inf_{\precdot}$ we have that $\im(q_s)$ is an initial segment of $\omega$.
        \item \label{J4} For every $s\in\Inf_{\precdot}$ we have $\dom(q_s)\subseteq \predom(\precdot_s)\cap B$.
        \item \label{J5} For every $s\in\Inf_{\precdot}(B)$ we have $\dom(q_s)\supseteq \predom(\precdot_s)\cap B$.
    \end{enumerate}
\end{prop}

\begin{proof}
    Given a condition $r\in \mathcal{Q}$, write $\Inf_{\precdot}^{r}$ for the set of all $s\in\Inf_{\precdot}$ such that $\im(r_s)\neq\emptyset$. Our proof will, for each $1\leq i\leq 5$, construct $q_i\in \mathcal{Q}$ and $B_i\in[\omega_n]^{<\omega}$ such that the items $(1),...,(i)$ hold, and then $q=q_5$, $B=B_5$ will satisfy the conclusion. Assume without loss of generality that $p$ is not the empty condition. We start by setting $q_1=p$ and $B_1=A$. We now set $q_2=q_1$ and $B_2=B_1\cup \bigcup(B_1\cap \omega)$, so that (\ref{J1}) and (\ref{J2}) both hold.

    To each $s\in\Inf_{\precdot}^{q_2}$ we associate the set $X_s=\max\{\im(q_2)_s\}\setminus\im(q_2)_s$ and write $j_s=|X_s|$. We now write $X_s=\{m^s_0,...,m^s_{j_{s-1}}\}$. Let $\{n^s_0,...n^s_{j_{s-1}}\}$ denote the first $j_s$ elements of $\omega\setminus\dom(q_2)_s$, and define $q_3=q_2\cup\{(n^s_i,s;m^s_i):s\in\Inf_{\precdot}^{q_2},i< j_s\}$ (note that by \Cref{Corollary: contains omega} the elements $n^s_i$ are elements of $\predom(\precdot_s)$) and $B_3=B_2$. By our choice of $m_i$ the functions $q_s$ are still injective, meaning $q_3\in \mathcal{Q}$ and (\ref{J1}) clearly holds. Trivially (\ref{J2}) still holds. By definition of $X_s$ and $q_3$ we have that $\im(q_3)_s=\im(q_2)_s\cup X_s$ is an initial segment of $\omega$, so (\ref{J3}) holds for all $s\in\Inf_{\precdot}^{q_2}=\Inf_{\precdot}^{q_3}$, and so the same holds for all $s\in\Inf_{\precdot}$ (for any $s\notin\Inf_{\precdot}^{q_3}$ we trivially have that $\im(q_3)_s$ is an initial segment of $\omega$).

    To every $s\in\Inf_{\precdot}^{q_3}$ we set $\{a^s_0,...,a^s_{l_s-1}\}=\dom(q_3)_s\setminus B_3$ for $l_s=|\dom(q_3)_s\setminus B_3|$. We now define $q_4=q_3$, $B_4'=B_4\cup\{a^s_i:s\in\Inf_{\precdot}^{q_4},i<l_s\}$ and $B_4=B_4'\cup \bigcup(B_4'\cap\omega)$. We have that (\ref{J2}) holds by definition of $B_4$, and since $q_4=q_3$ we have that (\ref{J1}) and (\ref{J3}) both hold. By choice of $a^s_i$ we have that (\ref{J4}) holds for all $s\in\Inf_{\precdot}^{q_3}$, and so the same holds for all $s\in\Inf_{\precdot}$ (for any $s\notin\Inf_{\precdot}^{q_3}$ we clearly have $\dom(q_3)\subseteq\predom(\precdot_s)\cap B$.

    Finally, for any $s\in\Inf_{\precdot}(B_4)$, write $\{x^s_0,...,x^s_{m_s -1}\}$ for the elements of the set $(\predom(\precdot_s)\cap B_4)\setminus\dom(q_4)_s$ and $\{y^s_0,...,y^s_{m_s -1}\}$ for the first $m_s-1$ elements of $\omega\setminus \im(q_4)_s$, and set $B_5=B_4$ and $q_5=q_4\cup\{(x^s_i,s;y^s_i):s\in\inf_{\precdot}(B),i<m_s-1\}$. We have that (\ref{J1}), (\ref{J2}) hold by induction. For $s\in\Inf_{\precdot}\setminus\Inf_{\precdot}(B_5)$ we have that (\ref{J3}) and (\ref{J4}) hold because $(q_4)_s=(q_5)_s$. and for $s\in\Inf_{\precdot}(B_5)$ by choice of $x^s_i,y^s_i$. To finish, (\ref{J5}) also holds by choice of $x^s_i$.
\end{proof}

\begin{proof}[Proof of \cref{fc}]
    Fix $A\in[\omega_n]^{<\omega}$. We will show that the set $D_A$ of conditions which force that $A$ is contained in a finite $\dot{\precdot^G}$-closed set $B\in[\omega_n]^{<\omega}$ is dense in $\mathcal{Q}$. Fix $p\in\mathcal{Q}$. Let $(B,q)$ be a pair given by applying \cref{const} to $(A,p)$, so in particular it satisfies the conditions of \cref{red}, hence $q$ forces that $B$ is closed. Since $B$ is finite and contains $A$, the result follows.
\end{proof}

\begin{cor}
\label{Corollary: omega_n}
    Given $n<\omega$, it is consistent with ZFC that $\VCcof(\omega,\omega,\omega_n)=n+1$.
\end{cor}

\begin{proof}
    Let $G\subseteq\mathcal{Q}$ be generic, then $V[G]$ and $V$ have the same cardinals by \cref{ccc} and $\precdot^G$ is an $(n+1)$-ordering system on $\omega_n$ with finite closures by \cref{fc}. The collection of finite $\precdot^G$-closed sets has VC-dimension $n+1$ by \cref{prop:OVC}, so we finish by \Cref{cor:closures imply bounded VC dimension}.
\end{proof}

\begin{que}
    Can the above construction be modified so that the closed $(n+1)$-initial segments of the orderings $\precdot^G$ are themselves $\precdot^G$-closed (as in \Cref{Theorem: closed initial segments})?
\end{que}

\section{Conclusions and questions}

Combining \Cref{monotonicity,lower1,prop:measurable,singular,regular,Corollary: omega_n} gives us the following.

\begin{thm}\label{thm:main}
    Suppose $\lambda\leq\theta\leq\kappa$ are infinite cardinals.
    \begin{enumerate}
        \item \label{enu:infinite VC} For $\theta^{+\omega}\leq \kappa$ we have $\VCcof(\lambda,\theta,\kappa)=\infty$.
        \item \label{enu:regular} For $\theta$ regular uncountable, if $\kappa=\theta^{+n}$ for some $n<\omega$ then $\VCcof(\lambda,\theta,\kappa)=n+1$.
        \item \label{enu:omega} For $\theta=\omega$, if $\kappa=\omega_n$ for some $n<\omega$ then it is consistent that $\VCcof(\lambda,\theta,\kappa)=n+1$ (note that in this case $\lambda=\omega$).
        \item \label{enu:singular} For $\theta$ singular, arbitrary $\kappa$,  and any $\lambda$ satisfying $\cof(\theta)<\lambda$ we have $\VCcof(\lambda,\theta,\kappa)=\infty$.
        \item \label{enu:measurable} Relative to the consistency of a measurable cardinal, it is consistent that there exists an uncountable cardinal $\theta$ of countable cofinality satisfying \\$\VCcof(\omega,\theta,\theta^{+n})=n+1$ 
        for every $n<\omega$. 
    \end{enumerate}
\end{thm}

\begin{que}
    Is there a model of ZFC and a natural number $n$ for which there is no $n+1$-ordering system on $\omega_n$ with finite closures? More generally, is it consistent that $\VCcof(\omega,\omega,\omega_n)>n+1$ for some $n<\omega$?
\end{que}

\begin{question}
    Can one construct cofinal VC-classes of finite subsets of an infinite set in a way which does not require an ordering system?
\end{question}

\printbibliography

\end{document}